\documentclass[11pt]{amsart}
\usepackage{amsmath,amssymb, xypic,verbatim, amscd,color}
\usepackage{graphicx}
\usepackage{colordvi}
\usepackage{mathdots}
\usepackage[colorlinks=true, pdfstartview=FitV,
linkcolor=cyan, citecolor=red, urlcolor=blue]{hyperref}
\usepackage{amsfonts,latexsym}

\numberwithin{equation}{section}

\usepackage{url, amsthm,xypic,amssymb,graphicx,color,enumerate,mathrsfs, amscd}
\usepackage{bm}
\usepackage[margin=0.8in]{geometry}

\usepackage{enumerate,bbm}
\usepackage[english]{babel}

\usepackage{ stmaryrd }
\usepackage{ulem}
\usepackage[all]{xy}

\usepackage{thmtools}
\usepackage{thm-restate}

\usepackage{hyperref}

\def\Gm{\mathbb{G}_m}
\def\Ga{\mathbb{G}_a}

\def\P{\mathbb{P}}
\def\C{\mathbb{C}}
\def\A{\mathbb{A}}
\def\Z{\mathbb{Z}}

\def\x{\times}
\def\ox{\otimes}
\def\a{\alpha}

\def\sdpLG{L^\ltimes_{pol}G}
\def\pLG{L_{pol}G}
\def\ppLG{L^{+}_{pol}G}
\def\tC{\widetilde{C_0}}

\def\cMG{\mathcal{M}_G}

\def\tlt0{(\tilde LT/T)_0}

\DeclareMathOperator{\ec}{Spec }

\newcommand{\mc}[1]{\mathcal{#1}}
\newcommand{\ol}[1]{\overline{#1}}

\newtheorem{thm}{Theorem}[section]
\newtheorem{prop}[thm]{Proposition}
\newtheorem{lemma}[thm]{Lemma}
\newtheorem{cor}[thm]{Corollary}

\theoremstyle{definition}
\newtheorem{definition}[thm]{Definition}
\theoremstyle{remark}
\newtheorem{rmk}{Remark}
\theoremstyle{notation}

\usepackage{pgf}
\usepackage{tikz}
\usetikzlibrary{arrows,automata} 
\begin{document}

 \title{Nodal Uniformization of $G$-bundles}
\author{ Pablo Solis}
\address{Department of Mathematics,
 Caltech, 1200 E California Blvd,
Pasadena, CA 91125}
\email{pablos.inbox@gmail.com}

\thanks{I would like to thank Constantin Teleman for suggesting the question of nodal uniformization. I have also had many useful conversations with Xinwen Zhu and I also thank Carlos Simpson for sharing his insight on this problem. Finally, I thank Swarnava Mukhopadhyay for pointing me to the work of Belkale and Fakhruddin and I thank Prakash Belkale for helpful comments and pointing out an error in an earlier version of this paper.}

\maketitle

\tableofcontents

\section{Introduction}
The uniformization theorem says that if $G$ is a split semisimple group over a field $k$ then any $G$ bundle on a smooth affine curve is trivial. In this form the result goes back to a 1967 result of Harder \cite{HarMR0225785} which proves it for $G$ bundles over a Dedekind domain. In more recent work, Beauville and Laszlo \cite{BeaLasMR1289330} and Drinfeld and Simpson \cite{DSMR1362973} generalized the result to triviality of bundles over families of smooth affine curves; see also Teleman \cite{TelMR1646586} for $k = \C$. There is also a version of uniformization for torsors for a non constant group scheme due to Heinloth. Recently Belkale and Fakhruddin have generalized the result to singular curves \cite{BelFak}. 

In this paper we first prove a special case of uniformization for nodal curves, theorem \ref{thm:trivNodalSect}. We apply \ref{thm:trivNodalSect} to construct a relative compactification of the moduli of $G$-bundle over a family of smooth curves degenerating to a nodal curve. Theorem \ref{thm:trivNodalSect} is implied by \cite[thm\;1.4]{BelFak} thus the main contribution of this work is to the application of compactification of moduli of $G$-bundles on nodal curves.

To explain the application we give some more background on uniformization. The families version of uniformization relates three important objects in geometry and representation theory: the loop group $L G$, the affine Grassmannian $Gr_G$ and the moduli stack $\cMG(C)$ of $G$ bundles on $C$. There is a sequence of morphisms
\[
L G \to Gr_G \to \cMG(C)
\]
and arguably {\it the} main corollary of Drinfeld and Simpson's uniformization theorem is that the morphism $Gr_G \to \cMG(C)$ is surjective. This leads to among other results a computation of $Pic(\cMG(C))$ and a proof that $\cMG(C)$ is irreducible. Typically one is in the setting of a family of proper or projective smooth curves $C \to S$ equipped with a principal $G$-bundle $E$ and the uniformization theorem comes in order to ensure $E$ is trivial on the complement of a section. 

In our setting we work over $\C$ and our base $S$ is a smooth curve with a special point $s_0 \in S$ and the curve $C \to S$ is smooth when restricted to $S- s_0$. The singular fiber $C_0$ has a unique node $p_0$. In general the total space $C$ could be smooth or $p_0$ could be a singular point. In the former case a ramified base change $S' \to S$ is necessary to ensure a section passing through $p_0$. Let $U_S$ denote the complement of such a section.

Theorem \ref{thm:trivNodalSect} shows when $G$ is simply connected that $G$ bundles on $U_S$ are trivial. We combine this with a degeneration $X$ of the loop group constructed in \cite{Solis}. The space $X$ is only a partial compactification but it gives an interesting compactification of the moduli stack of bundles on the nodal fiber $C_0$; see proposition \ref{p:mainApp}. Further we give a description of the boundary in terms of compactifications of finite dimensional groups. Specifically, over a fixed nodal curve $C_0$ the stack of bundles $\cMG(C_0)$ is a (non principal) fiber bundle with fiber $G$ over $\cMG(\tC)$. We explain how to compactify the fibers using an equivariant compactification $\ol{G}$ of $G$ yeilding a stack $\ol{\cMG(C_0)}$ which is universally closed. The compactification one obtains from $X$ is a union of stacks of the form $\ol{\cMG(C_0)}$ however only one component uses a compactification of $G$. The other components use compactifications of groups which appear as Levi factors for parahoric subgroups in the loop group. 

We do not know if the simply connectedness assumption can be dropped. It would be interesting to know if nodal uniformization holds for non-simply connected semisimple groups, or more generally for singular curves.

\section{Versions of Uniformization}
Here we state various versions of uniformization in part to review the literature and also in part to raise the question of what is the correct generality to pursue uniformization theorems. 

To utilize the result for moduli stacks of bundles one needs a uniformization theorem {\it in families}. In their 1994 paper on conformal blocks \cite{BeaLasMR1289330}, Beauville and Laszlo prove
\begin{lemma}[$SL_r$ Uniformization]\label{l:SLrCase}
Let $S$ be defined over an algebraically closed field of characteristic $0$. Let  $C \to S$ be a connected projective curve and $D \subset C$ a relatively ample Cartier divisor and set $C^* = C - D$. Then for any $S$-scheme $T$ and any $SL_r$-bundle $E$ on $C_T = C\x_S T$ there is a Zariski cover $T = \cup_i T_i$ such that $E|	_{C^*_{T_i}}$ is trivial.
\end{lemma}
\begin{proof}
\cite[lemma\;3.5]{BeaLasMR1289330}. The results is stated for smooth constant families $C \x S$ and for $D$ the image of a section but the proof works more generally. See also \cite[prop.\;3.2]{BelFak}. In outline, one applies induction on the rank of $E$. As $D$ is ample, after a large twist $E(n D)$ one can find a section non vanishing on $C^*$ yielding an exact sequence 
\[
0\to \mc{O}_{C'} \to E(n D)|_{C^*} \to E' \to 0.
\]
Each $C^*_{T_i}$ is affine (see paragraph before thm \ref{thm:trivNodalSect}) hence the sequence splits and by induction we are done.
\end{proof}

The generalization from $SL_r$ to other semi simple groups was handled by Drinfeld and Simpson. 
\paragraph{Smooth Setting}\label{smooth}
Let $S$ be a scheme and $C$ a proper smooth scheme over $S$ with connected geometric fibers of pure dimension $1$. Let $G$ be a split reductive group over $\Z$ and $B$ a Borel subgroup.

\begin{thm}[Smooth Uniformization]\label{thm:DSU}
Assume \ref{smooth}. Suppose further $G$ is semisimple, $E$ a $G$ bundle on $C$ and let $\sigma \colon S\to X$ be a section and set $U = C - \sigma(S)$. There is a faithfully flat base change $S'\to S$ of finite presentation such that $E$ becomes trivial on $U\x_S S'$. If $S$ is a scheme over $\Z[n^{-1}]$ with $n = |\pi_1(G(\C))|$ then $S'$ can be taken to be an etale base change.
\end{thm}
\begin{proof}
\cite[Thm\;3]{DSMR1362973}
\end{proof}
\begin{rmk}
In this paper we are primarily in the case $G$ is simply connected and defined over $\C$ in which case we can always take an etale base change in the theorem.
\end{rmk}

Theorem \ref{thm:DSU} is deduced by showing any $G$ bundle admits a reduction to a Borel:

\begin{thm}[Smooth $B$-structures] \label{thm:DSB}
Assume \ref{smooth} and $E$ a $G$ bundle on $C$. Then there is an etale base change $S' \to S$ such that $E$ admits reduction to $B$ on $C\x_S S'$.
\end{thm}

In subsection \ref{DSargument} we give a sketch of the proof of theorem \ref{thm:DSU} and how it is used to prove \ref{thm:DSU}. The argument uses the theory of Hilbert schemes together with the Riemann-Roch theorem for curves. 

Even though the statement of theorem \ref{thm:DSU} is very general the argument itself is in fact more general in the sense that various key ideas of the proof hold more generally. As evidence, Belkale and Fakhruddin have generalizations of \ref{thm:DSU},\ref{thm:DSB} for singular curves.

\paragraph{Singular Setting}\label{singular}
Let $S$ be any scheme and $C$ a proper, flat and finitely presented curve over $S$ and $G,B$ as in \ref{smooth}.

\begin{thm}[Singular $B$-structures]\label{BFB}
Assume \ref{singular} and $E$ a $G$ bundle on $C$. Then there is an etale base change $S' \to S$ such that $E$ admits reduction to $B$ on $C\x_S S'$.
\end{thm}

\begin{thm}[Singular Uniformization]\label{BFU}
Assume \ref{singular}. Let $D \subset C$ be a be a relatively ample effective Cartier divisor which is flat over $S$, and let $U=C-D$. Assume $G$ is semisimple and simply connected. There is an etale base change $S'\to S$ such that $E$ becomes trivial on $U\x_S S'$.
\end{thm}
Belkale and Fakhruddin also have a version of theorem \ref{BFU} that doesn't require $G$ to be simply connected but more assumptions are needed \cite[thm\;1.5]{BelFak}.

\subsection{$B$-structures on $G$-bundles and local triviality}\label{DSargument}
Here we give in outline the ideas of \cite{DSMR1362973}.

Let $S$ be a scheme and $X \to S$ a proper morphism. Let $G$ be a split reductive group and $B$ a Borel subgroup and $\pi \colon E \to X$ a principal $G$ bundle. For any $S$-scheme $T$ let $X_T = X\x_S T$ and $E_T = E\x_X X_T$ and $\pi_T \colon E_T \to X_T$.

A {\it $B$-reduction over $T$} is a section $\sigma$ of $E_T/B \to X_T$. Set $E^B_T = \sigma^*E_T$; this is $B$-bundle whose associated $G$ bundle is isomorphic to $E_T$.

Define $\Phi \colon Sch/S \to Set$ by setting $\Phi_E(T)$ to be the set of $B$-reductions of $E$ over $T$. By identifying $\sigma$ with its graph $\Gamma_\sigma \subset X_T \x_T E_T/B$ it follows that $\Phi_E$ is a subfunctor of the Hilbert scheme of subschemes of $X\x_S E/B$ and therefore representable by a scheme $\phi \colon M_E \to S$ of finite presentation over $S$. Also let $\mc{T}_{\pi/B}$ denote the relative tangent bundle of $E/B$ over $X$. So in summary we have:
\[
\xymatrix{E \ar[d]^{\pi}  &  \mc{T}_{\pi/B}\ar[d] & M_E \ar[d]^{\phi}\\ X_S \ar[r]^{\sigma} & E/B   & S}.
\]

For a point $s\in S$ let $\sigma$ be a $B$ reduction of $E_s$. Standard deformation theory shows that the obstruction to lifting $\sigma$ to an infinitesimal thickening of $X_s$ lies in $H^1(X_s, \sigma^*\mc{T}_{\pi/B})$. Therefor by setting 
\[
M^+_E =\{\sigma \in M_E| H^1(X_{\phi(\sigma)}, \sigma^*\mc{T}_{\pi/B}) = 0\}\] 
we get a scheme $M^+_E$ which is smooth over $S$.  Because any smooth morphism has etale local sections we obtain
\begin{prop}\label{p:phiMap}
Let $E$ be a $G$-bundle on $X$ and let $s\in S$. If $s$ lies in the image of $M^+_E$ then $E$ admits a $B$-reduction etale locally at $s$.
\end{prop}
Assume now $X = C$ is a smooth curve. Let $E^B$ be a $B$ reduction of $E$. For each positive root $\a$ we have a line bundle $E^B\x^\a \Gm$; set $d_\a(E^B) = \deg(E^B\x^\a \Gm).$ One can use the Riemann-Roch theorem to show when that if all $d_\a(E^B)$ are sufficiently negative then $E^B$ corresponds to a point of $M^+_E$. Drinfeld and Simpson prove
\begin{prop}\label{p:DSdegAlpha}
Let $C$ be a smooth projective curve over an algebraically closed field $k$ and $E$ a $G$-bundle on $C$. Then for any number $N$ there is a $B$-reduction $E^B$ such that $\deg_\a(E^B) < −N$ for all positive roots $\a$.
\end{prop}
\begin{proof}
\cite[prop\;3]{DSMR1362973}
\end{proof}
\begin{rmk}\label{r:d_a}
An analogue of this result is proved in the singular case in \cite{BelFak}.
\end{rmk}

The previous two propositions prove theorem \ref{thm:DSB}. The following result together with lemma \ref{l:SLrCase} proves theorem \ref{thm:DSU}.

\begin{prop}\label{p:UniFromGL2case}
Suppose $U \to S$ is an affine morphism and suppose any two $GL_2$ bundles are isomorphic on a cover of $S$ provided they have the same determinant. Let $G$ be semisimple and simply connected. Let $E$ be a $B$-bundle on $U$. Then $E$ reduces to a maximal torus $T\subset B$ and the associated $G$-bundle $E(G)$ is trivial on some cover of $S$.  
\end{prop}

\begin{proof}
We can assume $S$ and hence $U$ are affine. Let $B_u$ be the unipotent radical of $B$. As $B_u$ is a successive extension of $\Ga$s and $H^1(U,\Ga)=0$ there are no $B_u$ bundles on $U$. Thus $ i \colon H^1(U,B)  \to H^1(U,T)$ is injective. Also the inclusion $T \subset B$ provides a section to $i$. 

The key to reducing from general $G$ to the $GL_2$ case is to introduce the twisting of a $T$ bundle by pair $(\mc{L},\lambda)$ where $\mc{L}$ is line bundle and $\lambda \in \hom(\Gm,T)$ is a 1 parameter subgroup. Let $\cup_i U_i$ be a Zariski cover of $U$ that trivializes $E$ and $\mc{L}$. Let $t_{ij}$ be the transition functions for $\mc{L}$. The twisted bundle is denoted $E\ox\mc{L}(\lambda)$ and has transition functions given by those of $E$ multiplied by $U_i \cap U_j \xrightarrow{t_{ij}} \Gm \xrightarrow{\lambda} T$. 

Any $T$ bundle on $U$ is obtained by a finite number of twistings of the trivial bundle. Moreover, because $G$ is simply connected $\hom(\Gm,T)$ is generated by the simple co-roots $\a^\vee$ of $G$. Thus we are reduced to showing if $E' = E\ox \mc{L}(\a^\vee)$ then $E(G),E'(G)$ are isomorphic locally over $S$. In fact we can show this with $G$ replaced by the group $G'$ generated by $G_\a$ and $T$ where $G_\a\cong SL_2$ is a principal $SL_2$ subgroup corresponding to $\a$. It is routine to verify that $G' \cong SL_2 \x T'$ or $G' = GL_2 \x T'$ for $T'$ a torus. Clearly this reduces to the question to $G = SL_2,GL_2$ and, in the $GL_2$ case, twisting by $\a^\vee$ doesn't change the determinant thus we are done by hypothesis.   
\end{proof}
\begin{rmk}
It is explained in \cite{DSMR1362973} how to reduce the general case to the non simply connected case. Also the hypothesis of the proposition follow in the smooth curve case from lemma \ref{l:SLrCase}.
\end{rmk}

We should also mention the work of Heinloth \cite{HeinMR2640041}. He has generalized the results of Drinfeld and Simpson to torsors for nonconstant semisimple group schemes $\mc{G}$ over a smooth curve $C$. His approach is quite different from the ideas presented here. The key in Heinloth's approach is to use that the morphism $Gr_{\mc{G}} \to \mc{M}_{\mc{G}}(C)$ is smooth.  

Given that uniformization works for the constant group scheme over singular curves it is tempting to wonder about uniformization for non constant group schemes over singular curves.

\section{Uniformization and $B$-reductions for Nodal Curves}
Here we quickly establish results of uniformization and reduction to a Borel for a fixed nodal curve and for a family of smooth curves degenerating to a nodal curve. The results in this section are either special cases or implied by results in \cite{BelFak} which deals with arbitrary singular curves. The proofs given here are streamlined for nodal curves which is sufficient for our main application in section \ref{s:compact} on compactifications of moduli spaces. 

\begin{lemma}\label{l:G(affNodal)}
Let $C = \ec A$ be an affine nodal curve and $p,q \in C$ distinct smooth points. Let $G$ be a semisimple group and $g\in G(\C)$. Then there is a $\gamma\in G(A)$ such that $\gamma(p) = 1$ and $\gamma(q) = g$.
\end{lemma}

\begin{proof}
Let $G^{sc}$ be the universal cover of $G$. Let $g'$ be a lift of $g \in G(\C)$ to $G^{sc}(\C)$. Verifying the lemma for $G^{sc}$ and $g'$ implies it for $G$ and $g$. So we can assume $G = G^{sc}$ and $G = \langle U_\a \rangle_{\a \in R}$ where $R$ is the set of roots and $U_\a$ the corresponding root subgroups. Using isomorphisms $\mathbb{G}_a \xrightarrow{\phi_\a} U_\a$ write $g = \prod_i \phi_{\a_i}(t_i)$. Let $m_p,m_q$ be the maximal ideals of $p,q$. As $p \neq q$ there is an $f\in m_p$ such that $f(q) = 1$. We can take $\gamma = \prod_i \phi_{\a_i}(t_i f).$
\end{proof}

\begin{lemma}\label{l:affineNodal}
Let $C$ be a nodal affine curve and $G$ a semisimple group. Then any $G$ bundle on $C$ is trivial.
\end{lemma}
\begin{proof}
Use induction on the number of nodes; the base case being handled by \cite{HarMR0225785}. Assume $C$ has $n$-nodes, $\nu \colon C' \to C$ is the partial normalization at $n$th node $x$ and let $p,q = \nu^{-1}(x)$. By induction, $E' = \nu^*E$ is trivial and obtained by an identifying isomorphism $\phi \colon E'_p \to E'_q$. Fixing a global trivialization of $E'$ allows us to consider $\phi \in G$. We can change $\phi$ by applying an element of $Aut(E') = G(\C[C'])$. Therefore by lemma \ref{l:G(affNodal)}, we can take $\phi = id$. 
\end{proof}

\begin{cor}\label{c:NodalBstr}
Let $C$ be nodal projective curve and $G$ a reductive group and $B$ a Borel subgroup. Any $G$-bundle reduces to $B$ and consequently any $G$-bundle is Zariski locally trivial.
\end{cor}

\begin{proof}
Write $G$ as an extension $1\to G^{ss} \to G \to T\to 1$ with $G^{ss}$ semisimple and $T$ a torus. This gives rise to an exact sequence
\begin{equation}\label{eq:exactH1}
H^1_{et}(C,G^{ss}) \to H^1_{et}(C,G) \to H^1_{et}(C,T).
\end{equation}
Let $E$ be a $G$ bundle and $E(T)$ the associated $T$ bundle. As $Pic(C)$ is generated by $C^{sm}$ we have $E(T)$ is trivial $U = C-\{p_1, \dotsc, p_n\}$ with $p_i$ smooth. Hence by \eqref{eq:exactH1}, $E|_U$ comes from an $G^{ss}$-bundle and hence trivial on $U$. In particular $E$ has a reduction to $B$ on $U$. As $E/B \to C$ is projective, the valuative criterion implies the reduction extends over the $p_i$. 
\end{proof}
\begin{rmk}
If $G$ is any linear algebraic group then by using the exact sequence $1\to G_u \to G \to G_{red} \to 1$ where $G_u$ is the unipotent radical one can show any $G$-bundle on $C$ is Zariski locally trivial. But we do not use this.
\end{rmk}

\subsection{Nodal Degeneration}
We assume the following in this section
\paragraph{Nodal Degeneration}\label{mycase}
Let $S$ be a smooth curve with a base point $s_0 \in S$ and $C$ a proper scheme over $S$ with connected geometric fibers of pure dimension $1$. We further assume $C \to S$ is finitely presented and $C$ is smooth on $S -s_0$ and $C_0:=C_{s_0}$ is a nodal curve with a unique node $p_0$. These assumptions imply that $C_0 - p_0$ is smooth, affine and either irreducible or the disjoint union of two smooth affine curves. 

Let $\ec A$ be an etale neighborhood of $s_0 \in S$ and $t \in m_{s_0}$ a local coordinate at $s_0$. A neighborhood of the node is etale locally $\ec A[x,y]/(x y - t^k)$. In the special case $k = 1$ we need to pass to a double cover $t \mapsto t^2$ to ensure there is a section passing through $p_0$. With that special case in mind we let $D_{S'} \subset C\x_S S'$ be the image of a section passing through $p_0$.

We now prove analogues of Drinfeld and Simpson's theorems \ref{thm:DSB},\ref{thm:DSU}. Proposition \ref{p:BstrNodal} which addresses $B$-reductions is not new, it is covered by Belkale and Fakhruddin's theorem \ref{BFB}; however because we are only interested in the nodal case we can give a more direct argument that establishes the result as a corollary of prop \ref{p:DSdegAlpha}. In turn this gives a quick proof of theorem \ref{thm:trivNodalSect}; we note theorem \ref{thm:trivNodalSect} is implied by \cite[Thm\;1.4]{BelFak}.

\begin{prop}\label{p:BstrNodal}
Assume \ref{mycase} and let $G$ be reductive and $E$ a $G$ bundle on $C$. Then there is an etale base change $S' \to S$ such that $E$ admits reduction to $B$ on $C\x_S S'$.
\end{prop}
\begin{proof}
As $C$ is smooth on $S - s_0$ it suffices to verify proposition \ref{p:phiMap} at $s_0$. This will follow if there are $B$-reductions satisfying \ref{p:DSdegAlpha} on the normalization $\nu \colon \tC \to C_0$ that descend to $C_0$. 

By proposition \ref{p:BstrNodal}, $G$ bundles are trivial in a Zariski neighborhood of $p_0$ hence any two $G$-bundles are isomorphic on an open set of $p_0$ and a $B$-reduction on one induces one on the other changing $\deg_\a(E_B)$ by a bounded amount; see paragraph before \ref{p:DSdegAlpha} for the definition of $\deg_\a$. Thus we can reduce to $E$ being the trivial bundle on $C_0$.

Let $F = (\nu^*E)^B$ be a $B$-reduction of $\nu^*E$. If $F$ is trivial on an open set containing $\{p_1,p_2\}=\nu^{-1}(p_0)$ then $F$ descends to $C_0$. Also $F = \sigma^* G$ for a map $\sigma \colon \tC \to G/B$, as $\nu^*E$ is trivial. Moreover $G \to G/B$ is trivial over the image $V$ of $B^-B$. We claim there is $g \in G$ such that $\tC \xrightarrow{\sigma} G/B \xrightarrow{g} G/B$ sends $p_i$ to $V$; the resulting $B$-reduction descends to $C_0$. To prove the claim translate in $G/B$ so  $\sigma(p_1)$ is the base point. Let $g\in G$ be any lift of $\sigma(p_2)$. As $B^-B \cap  B^-B g^{-1} \neq \emptyset$, there is an $h \in B^-B$ such that $h \sigma(p_2) \in V$.

So $F$ descends to $C_0$; call the resulting $B$-bundle $E^B$. Because $F$ is trivial in a neighborhood of $p_1,p_2$ we have $\nu^* (E^B \x^\a \Gm) = F \x^\a \Gm$ and $\deg L = \deg \nu^* L$ for any line bundle hence the result follows from \ref{p:DSdegAlpha}.
\end{proof}

We can now prove theorem \ref{thm:trivNodalSect}. We note again that theorem 1.4 from \cite{BelFak} implies theorem 3.5 below, since a multiple of the section is Cartier, and ample. Specifically, we can directly compute the Picard group in a neighborhood of the singularity $p_0$. It suffices to do this etale locally as for any scheme $Pic_{et}\cong Pic_{Zar}$ (see e.g. \cite[cor.\;11.6]{Milne}). It is well known $Cl(A[x,y]/(x y - t^k)) =\Z/k$ hence a large multiple is Cartier.

\begin{thm}\label{thm:trivNodalSect}
Assume \ref{mycase} and let $G$ be semisimple and let $E$ be a $G$ bundle on $C$. Let $D$ be a section passing through $p_0$ and let $U = C - D$. Then there is an etale base change $S' \to S$ such that $E$ is trivial on $U \x_S S'$.
\end{thm}
\begin{proof}
We can apply \ref{thm:DSU} on $C_{S - s_0}$ thus we need only prove the result for an etale neighborhood of $s_0$. Thus we can assume $S$ is affine. Moreover replacing $S$ with an etale cover $S'$ by \ref{p:BstrNodal} we can assume $E$ reduces to a $B$-bundle. 

We aim to apply \ref{p:UniFromGL2case} and first check the affine hypothesis. To show the morphism $U_S \to S$ is affine if suffices to show $U_S$ is affine. Some multiple $n D$ is Cartier, if it is ample then $U_S$ can be realized as a closed subscheme of the complement of a hyperplane section in some $\P^n_S$. The map $C \to S$ is finitely presented and proper so by \cite[9.6.4]{MR0217086} it suffices to check $n D$ is ample on each fiber.  This is immedite on $C_{S - s_0}$ and on $C_{s_0}$ the restriction to each component of $C_{s_0}$ is positive hence ample. 

To apply \ref{p:UniFromGL2case} it remains check the case $G=SL_2$ and to check that any two $GL_2$ bundles with the same determinant are isomorphic Zariski locally on $S'-s_0$. This last step is handled by \ref{l:SLrCase}. Note to apply \ref{l:SLrCase} is essential that some multiple $nD$ is relatively ample. The $GL_2$ case is proved mutatis mutandis.
\end{proof}

\section{Compactification}\label{s:compact}
Recall the setup in \ref{mycase}. I'll abbreviate $C_{s_0}$ to $C_0$. To warm up we first describe how to use a $G \x G$ equivariant compactification $\ol{G}$ of $G$ to compactify $\cMG(C_0)$. 

If $p_1, \dotsc, p_n\in C$ then $\cMG^{p_1, \dotsc, p_n}(C)$ denotes the stack $G$ bundles $E$ equipped with trivializations $\tau_i \in E_{p_i}$. Let $q \in C_0$ be the node, $\nu \colon \tC \to C_0$ the normalization and $\nu^{-1}(q) = \{p_1,p_2\}$. Then $\cMG^q(C) \to \cMG(C)$ is a $G$ bundle and $\cMG^{p_1,p_2}(\tC) \to \cMG(\tC)$ is a $G \x G$ bundle and we have an equivalence  $\cMG^q(C_0) \cong \cMG^{p_1,p_2}(\tC)$ and a diagram
\[
\xymatrix{ G \ar[d]^{\Delta}\ar[r] & \cMG^q(C_0)\ar[r]^{f_q}\ar[d]^{\cong} & \cMG(C_0)\ar[d]^{\nu^*} \\
G\x G\ar[r] & \cMG^{p_1,p_2}(\tC)\ar[r]^{f_{1,2}} & \cMG(\tC)
}
\]
Here $f_q,f_{1,2}$ are principal bundles for $G$,$G\x G$ whereas $\nu^*$ is a fiber bundle with fiber the homogeneous space $\frac{G \x G}{\Delta(G)}$. In particular, $\nu^*$ does not make $\cMG(C_0)$ into a {\it principal} $G$ bundle over $\cMG$. Indeed, if we try to make $g \in G$ act on $E \in \cMG(C_0)$ we could first lift and $E$ to $(E,\tau_q) \in \cMG^q(C_0)$ then represent $(E,\tau_q) \cong (\nu^*E, \tau_1,\tau_2)$ and then take $g E = f_q(\nu^*E, \tau_1 g ,\tau_2)$. This is not well defined because if we instead chose the lift $(\nu^*E, \tau_1h,\tau_2h)$ then $g E = f_q(\nu^*E, \tau_1 h g ,\tau_2 h) = f_q(\nu^*E, \tau_1 h g h^{-1} ,\tau_2)$ which in general is not equal to $f_q(\nu^*E, \tau_1 g ,\tau_2)$.

In any case we have 
\[
\cMG(C_0) \cong \cMG^{p_1,p_2}(\tC)\x^{(G \x G)} \frac{G \x G}{\Delta(G)}.
\]
and we obtain a compactifcation as
\[
\ol{\cMG(C_0)}:= \cMG^{p_1,p_2}(\tC)\x^{(G \x G)} \ol{G}.
\]
\begin{lemma}
The stack $\ol{\cMG(C_0)}$ is universally closed over $\ec \C$.
\end{lemma}
\begin{proof}
The map $\ol{\cMG(C_0)} \to \ec \C$ factors as 
\[
\ol{\cMG(C_0)} \xrightarrow{c} \cMG(\tC) \to \ec \C.
\] So we just need that $c$ is universally closed. So assume we have a diagram 
\[
\xymatrix{\ec \C((t)) \ar[d] \ar[r]^{l} &  \ar[d] \ol{\cMG(C_0)} & \cMG^{p_1,p_2} \ar[dl]^{f}\\ \ec \C[[t]] \ar[r]^{j}  & \cMG(\tC) & }
\]
Denote by $t$ the closed point of $\ec \C[[t]]$ and let $W$ be an etale neighborhood of $j(t)$ over which $f$ is a trivial fibration. Then $c$ is trivial over $W$ and to extend $l$ is to extend a map $\ec \C((t)) \to \ol{G}$ to $\ec \C[[t]] \to \ol{G}$ which is always possible.
\end{proof}
The generalization to families of curves is more complicated. One must pass from $G$ to the loop group $L G$. We briefly review loop groups.  Let $\mathbf{Aff}_\C$ denote the category of $\C$-algebras, $\mathbf{Set}$ the category of sets and $\mathbf{Grp}$ the category of groups. Let $G$ be an affine algebraic group over $\C$.
\begin{definition}
The loop group $LG\colon \mathbf{Aff}_\C \to \mathbf{Grp}$ is the functor given by $LG(R):= G(R((z)))$ where $R((z))$ is the ring of formal Laurent series with coefficients in $R$.
\end{definition}
It is known that $LG$ is represented by an ind-scheme; an increasing union of infinite dimensional schemes. Elements $g(z) \in LG(R)$ are called loops. We also have positive loops $L^+G(R):= G(R[[z]])$ and polynomial versions $\pLG(R) = G(R[z^\pm])$ and $\ppLG = G(R[z])$. In fact a more relevant group is the semidriect product $\Gm^{rot} \ltimes \pLG$. This is defined by letting $u \in \Gm^{rot}$ act on a loop $g(z)$ by $g(z)\xrightarrow{u} g(u z)$. We abbreviate $\Gm^{rot} \ltimes \pLG$ to $\sdpLG$.

We consider $\sdpLG$ as a $(\sdpLG)^{\x 2}$ space via left and right multiplication. In \cite{Solis} a $(\sdpLG)^{\x 2}$ equivariant partial compactification $X$ of $\sdpLG$ is constructed inspired by the wonderful compactification of an adjoint group. It fits into a diagram 
\begin{equation}\label{eq:LGdegen}
\xymatrix{\sdpLG \ar[r] \ar[d] & X\ar[d] & \ar[l]\partial X \ar[d]\\
1 \ar[r] & \A^1/\Gm & \ar[l] 0/\Gm}
\end{equation}
In general $X$ is an ind-DM stack; in some cases it is actually an ind-scheme. In \cite{Solis} an explicit description is given to $\partial X$. To set up the description let $G$ be simple, simply connected or rank $r$, $T\subset G$ a maximal torus let $\a_0, \dotsc, \a_r$ be the simple roots of $LG$. To each root $\a_i$ we can associate a {\it maximal} parahoric subgroup $P_i \subset L G$ and moreover we can pick a Levi decomposition $P_i = L_i U_i$.

\begin{prop}\label{p:descriptFibers}
Let $G$ be simple, simply connected or rank $r$. The boundary $\partial X$ is a union of $r+1$ components $X_0 \cup \dotsb \cup X_r$ which can be labeled such that $X_i$ is a fibration over a product of flag varieties $\pi_i \colon X_i \to (LG/P_i)^{\x 2}$. For $i\neq 0$ the fibers of $pi_i$ are the wonderful compactification of $L_{i,ad}$ and the fiber of $pi_0$ is a $G^{\x2}$-equivariant compactification of $G$.
\end{prop}

The diagram \eqref{eq:LGdegen} suggests viewing $X$ as a degeneration of $\sdpLG$. By taking an appropriate quotient we turn $X$ into a degeneration of the affine grassmannian $Gr_G$. Specifically let $H = \ppLG$ and let $\mc{H}$ be the sheaf of groups over $\A^1/\Gm$ given by sections of $H \x \A^1/\Gm \to \A^1/\Gm$. Then there is an evaluation map $\mc{H} \xrightarrow{ev_0} H$ and let $\mc{H}_1 = \ker ev_0$. This acts on the right of $X$ and we let $X^{Gr} = X/\mc{H}_1$. The generic fiber is now $Gr_G$ and the special fiber remains unchanged and retains a left action of $L G$ on the generic fiber and an $LG\x LG$ action on the special fiber. 

Now we return to our curve $C\to S$. We take $S$ to be affine and $t \in m_{s_0}$ a local coordinate at $s_0$. We view $t$ as a section of $\mc{O}_S(-s_0)$ which we consider as a morphism to $\A^1/\Gm$. Using this morphism we pull back $X^{Gr}$ to obtain an ind-stack over $S$. Let $D$ be a section passing through $p_0$ and $U= C-D$. 

We now define a fiber wise set theoretic action of $G(U)$ on $X^{Gr}$. For each $g \in G(U)$ and $s \in S$ we can Taylor expand the restriction $g_s \in G(U_s)$ in a formal neighborhood of $D_s$. For $s \neq s_0$ this gives an element in $LG$ and for $s=s_0$ we obtain an element in $LG \x LG$. In this way we get an action of $G(U)$ on $X^{Gr}$.

It is desirable to have a more functorial construction which would realize $G(U)$ acting algebraically on $X^{Gr}$. Currently we can prove:
\begin{prop}\label{p:mainApp}
Let $G$ be simple, simply connected of rank $r$. For each $s \in S$ the group $G(U_s)$ acts algebraically on $X^{Gr}_s$. For $s \neq s_0$ the stack quotient $G(U_s)\backslash X^{Gr}$ is $\cMG(C_s)$. For $s=s_0$ we have the union $G(U_{s_0})\backslash X_{s_0}^{Gr} = \bigcup_{i=0}^r G(U_{s_0})\backslash  X_i$ and 
\[
G(U_{s_0})\backslash  X_0 \cong \ol{\cMG(C_0)}:= \cMG^{p_1,p_2}(\tC)\x^{(G \x G)} \ol{G}.
\]
\end{prop}
\begin{proof}
The first statement that the action is fiberwise algebraic holds because the original space $X$ has a fiberwise algebraic action of $L G$. The identification $G(U_{s_0})\backslash  X^{Gr}_s \cong \cMG(C_s)$ follows from Drinfeld and Simpson's uniformization theorem. 

For the final statement, consider the moduli space $\mc{M}_G^U(C_0)$ of pair $(E,\tau)$ where $E$ is a bundle, $\tau$ is a trivialization of $E$ over $C_0 - p_0$. Fixing an isomorphism $\widehat{\mc{O}}_{p_0} \cong \C[[x,y]]/x y$ yields $\mc{M}_G^U(C_0) \cong\frac{L_xG \x L_yG}{G[[x,y]]/x y}$. 

Moreover defining $N_x = \ker(G[[x]] \xrightarrow{x\to 0} G)$ and $N_y$ similarly we have 
\[
 \frac{L_xG \x L_yG}{G[[x,y]]/x y} = \frac{L_xG \x L_yG}{\Delta(G)\ltimes (N_x \x N_y)} 
\]
which realizes $\mc{M}_G^U(C_0)$ as a $G^{\x 2}/\Delta(G) \cong G$ fibration over $Gr_G^{\x 2}$. By proposition \ref{p:descriptFibers} we have
\[
X_0 = \mc{M}_G^U(C_0) \x^{G \x G} \ol{G}.
\]
Now we give a different presentation of $\mc{M}_G^U(C_0)$. Namely consider $\mc{M}^{U,p_1,p_2}_G(\tC)$ which consists of tuples $(E,\tau, t_1,t_2)$ where $E,\tau$ are as before and $t_i$ are framing of $E$ at $p_i$. The pair $(t_1,t_2)$ means $E$ descends to a bundle on $C_0$ with a single framing $t$ at $p_0$ giving rise to the moduli space $\mc{M}_G^{U,p_0}(C_0)$. We have
\begin{align*}
\mc{M}_G^{U}(C_0) &= \mc{M}_G^{U,p_0}(C_0)/G = \mc{M}^{U,p_1,p_2}_G(\tC) \x^{G \x G} G^{\x 2}/\Delta(G)\\
X_0 &\cong \mc{M}^{U,p_1,p_2}_G(\tC) \x^{G \x G} \ol{G}\\
G(U_{s_0})\backslash  X_0 &\cong \cMG^{p_1,p_2}(\tC)\x^{(G \x G)} \ol{G}.
\end{align*}
\end{proof}
 We can give a similar description for the other components $G(U_{s_0})\backslash  X_i$ with $i\neq 0$. For each $i\neq 0$ define a sheaf of groups $\mc{G}_i$ on $C_0$ which has $\mc{G}_i(\widehat{\mc{O}}_q) = G(\widehat{\mc{O}}_q)$ for $q \neq p_0$ and $\mc{G}_i(\widehat{\mc{O}}_{p_0}) = \Delta(L_{i,ad}) \ltimes (U_i \x U_i)$ where $L_{i,ad}$ is the adjoint form of $L_i$.
 
 Let $\mc{M}_{\mc{G}_i}(C_0)$ denote the moduli stack of torsors for $\mc{G}_i$ on $C_0$. As $\mc{G}_i$ is the sheaf of groups associated to the constant group scheme away from $p_0$, all $\mc{G}_i$ torsors are just $G$-bundle on $U_{s_0} = C_0 - p_0$ and are trivial. In particular we have
 \[
 \mc{M}_{\mc{G}_i}(C_0) = G(U_0)\backslash LG \x LG/\mc{G}_i(\widehat{\mc{O}}_{p_0}).
 \]
 Arguing as in the previous proposition we can also obtain 
 \[
G(U_{s_0}) \backslash X_i \cong \mc{M}_{\mc{G}_i}^{p_1,p_2} \x^{L_i \x L_i} \ol{L_{i,ad}}.
 \]
 
In particular the special fiber $G(U_{s_0})\backslash X_{s_0}^{Gr}$ in proposition \ref{p:mainApp} is a union of components each of which is universally closed.

A construction of a relative compactification which is not only fiberwise is work in progress.

\bibliographystyle{plain} 

\bibliography{nodalUniformization}

\def\cprime{$'$}
\begin{thebibliography}{1}

\bibitem{MR0217086}
Grothendieck A.
\newblock \'elements g\'eom\'etrie alg\'ebrique. {IV}. \'etude locale des
  sch\'emas et des morphismes . {III}.
\newblock {\em Inst. Hautes \'Etudes Sci. Publ. Math.}, (28):255, 1966.

\bibitem{BeaLasMR1289330}
Arnaud Beauville and Yves Laszlo.
\newblock Conformal blocks and generalized theta functions.
\newblock {\em Comm. Math. Phys.}, 164(2):385--419, 1994.

\bibitem{BelFak}
Prakash Belkale and Najmuddin Fakhruddin.
\newblock Triviality properties of principal bundles on singular curves.
\newblock {\em arxiv 1509.06425}.

\bibitem{DSMR1362973}
V.~G. Drinfel{\cprime}d and Carlos Simpson.
\newblock {$B$}-structures on {$G$}-bundles and local triviality.
\newblock {\em Math. Res. Lett.}, 2(6):823--829, 1995.

\bibitem{HarMR0225785}
G{\"u}nter Harder.
\newblock Halbeinfache {G}ruppenschemata \"uber {D}edekindringen.
\newblock {\em Invent. Math.}, 4:165--191, 1967.

\bibitem{HeinMR2640041}
Jochen Heinloth.
\newblock Uniformization of {$\mathcal{G}$}-bundles.
\newblock {\em Math. Ann.}, 347(3):499--528, 2010.

\bibitem{Milne}
J.S. Milne.
\newblock Lectures on \'etale cohomology.
\newblock {\em www.jmilne.org/math/CourseNotes/LEC.pdf}.

\bibitem{Solis}
P.~Solis.
\newblock A wonderful embedding of the loop group.
\newblock {\em arxiv 1208.1590}.

\bibitem{TelMR1646586}
Constantin Teleman.
\newblock Borel-{W}eil-{B}ott theory on the moduli stack of {$G$}-bundles over
  a curve.
\newblock {\em Invent. Math.}, 134(1):1--57, 1998.

\end{thebibliography}

\end{document}